\documentclass[reqno,b5paper]{amsart}

\usepackage{amsmath}
\usepackage{amssymb}
\usepackage{amsthm}
\usepackage{enumerate}
\usepackage[mathscr]{eucal}
\usepackage{eqlist}

\theoremstyle{plain}
\newtheorem{theorem}{Theorem}[section]
\newtheorem{proposition}[theorem]{Proposition}

\newtheorem{corollary}[theorem]{Corollary}
\theoremstyle{definition}
\newtheorem{definition}{Definition}[section]

\theoremstyle{remark}
\newtheorem{example}{Example}
\numberwithin{equation}{section}

\begin{document}

\title[A characterization of the discontinuity point set]%
{A characterization of the discontinuity point set of strongly separately continuous functions on products}
\author[Olena Karlova \and Volodymyr Mykhaylyuk]%
{Olena Karlova* \and Volodymyr Mykhaylyuk**}

\newcommand{\acr}{\newline\indent}

\address{\llap{*\,} Chernivtsi National University\acr
                   Kotsyubyns'koho 2\acr
                   58 012 Chernivtis\acr
                   UKRAINE}
\email{maslenizza.ua@gmail.com}

\address{\llap{**\,}Chernivtsi National University\acr
                   Kotsyubyns'koho 2\acr
                   58 012 Chernivtis\acr
                   UKRAINE}
\email{vmykhaylyuk@ukr.net}

\subjclass[2010]{Primary 54C08, 54C30; Secondary  26B05}
\keywords{strongly separately continuous function, separately continuous function}

\begin{abstract}
We study properties of strongly separately continuous mappings defined on subsets of products of topological spaces equipped  with the topology of pointwise convergence. In particular, we give a necessary and sufficient condition for a strongly separately continuous mapping to be continuous on a product of an arbitrary  family of topological spaces. Moreover, we charac\-terize the discontinuity point set of strongly separately continuous function defined on a subset of countable product of finite-dimensional normed spaces.
\end{abstract}

\maketitle

\section{Introduction}

In 1998 Omar Dzagnidze~\cite{Dzagnidze} introduced a notion of  a strongly separately continuous function $f:\mathbb R^n\to\mathbb R$. Namely, he calls a function $f$ {\it strongly separately continuous} at a point $x^0=(x_1^0,\dots,x_n^0)\in\mathbb R^n$ if for every $k=1,\dots,n$ the equality
$\lim\limits_{x\to x^0}|f(x_1,\dots,x_k,\dots,x_n)-f(x_1,\dots,x_k^0,\dots,x_n)|=0$ holds. \mbox{Dzagnidze} proved that a function $f:\mathbb R^n\to\mathbb R$ is strongly separately continuous at $x^0$ if and only if $f$ is continuous at $x^0$.

Extending the investigations of Dzagnidze, J.~\v{C}in\v{c}ura, T.~\v{S}al\'{a}t and T.~Visnyai in \cite{CSV} consider strongly separately continuous functions defined on the space $\ell_2$ of sequences $x=(x_n)_{n=1}^\infty$ of real numbers such that $\sum\limits_{n=1}^\infty x_n^2<+\infty$ endowed with the standard metric $d(x,y)=(\sum\limits_{n=1}^\infty (x_n-y_n)^2)^{1/2}$. In particular, the authors gave an example of a strongly separately continuous everywhere discontinuous function $f:\ell_2\to\mathbb R$.

Recently, T.~Visnyai in \cite{TV} continued to study properties of strongly separately continuous functions on $\ell_2$ and constructed a strongly separately continuous function $f:\ell_2\to\mathbb R$ which belongs to the third Baire class and is not quasi-continuous at every point. Moreover, T.~Visnyai gave a sufficient condition for strongly separately continuous function to be continuous on $\ell_2$.

In this paper we study strongly separately continuous functions defined on subsets of a product $\prod\limits_{t\in T} X_t$ of topological spaces $X_t$ equipped with the Tychonoff topology of pointwise convergence.
In Section 2 we introduce a notion and give the simplest properties of $\mathcal S$-topology which is tightly connected with  strongly separately continuous functions. In the third section we give a necessary and sufficient condition for a strongly separately continuous mapping to be continuous on a product of an arbitrary  family of topological spaces.
In Section 4 we find a necessary condition on a set to be the discontinuity point set of a strongly separately continuous mapping defined on a product of topological spaces. Finally, in the fifth section we describe the discontinuity point set of a strongly separately continuous function defined on a subset of a countable product of finite-dimensional normed spaces.

\section{A notion and properties of $\mathcal S$-topology}

Let $X=\prod\limits_{t\in T} X_t$ be a product of a family of sets $X_t$ with $|X_t|>1$ for all $t\in T$.
If $S\subseteq S_1\subseteq T$, $a=(a_t)_{t\in T}\in X$, $x=(x_t)_{t\in S_1}\in \prod\limits_{t\in S_1}X_t$, then we denote by $a_S^x$ a point $(y_t)_{t\in T}$, where
$$
y_t=\left\{\begin{array}{ll}
             x_t, & t\in S, \\
             a_t, & t\in T\setminus S.
           \end{array}
\right.
$$
In the case $S=\{s\}$ we shall write $a_s^x$ instead of $a_{\{s\}}^x$.

If $n\in\mathbb N$, then
$$
\sigma_n(x)=\{y=(y_t)_{t\in T}\in X: |\{t\in T:y_t\ne x_t\}|\le n\}
$$
and
$$
\sigma(x)=\bigcup\limits_{n=1}^\infty \sigma_n(x).
$$
Each of the sets of the form $\sigma(x)$ for an $x\in X$ is called {\it a $\sigma$-product of $X$}.

We denote by $\tau$ the Tychonoff topology on a product $X=\prod\limits_{t\in T}X_t$ of topological spaces $X_t$. If $X_0\subseteq X$, then the symbol $(X_0,\tau)$ means the subspace $X_0$ equipped with  the Tychonoff topology induced from $(X,\tau)$.

If $X_t=X$ for all $t\in T$ then the product $\prod\limits_{t\in T}X_t$ we also denote by $X^{\mathfrak m}$, where $\mathfrak m=|T|$.

\begin{definition}\label{def:snow-flake}
  {\rm A set $A\subseteq \prod\limits_{t\in T} X_t$ is called {\it $\mathcal S$-open} if $\sigma_1(x)\subseteq A$ for all $x\in A$.}
\end{definition}

We remark that the definition of an $\mathcal S$-open set develops the definition of a set of the type ($P_1$) introduced in \cite{CSV}.

Let $\mathcal S(X)$ denote the collection of all $\mathcal S$-open subsets of $X$. We notice that $\mathcal S(X)$ is a topology on $X$. We will denote by  $(X,\mathcal S)$ the product $X=\prod\limits_{t\in T}X_t$ equipped with the topology $\mathcal S(X)$.

The proof of the following is straightforward.

\begin{proposition}\label{prop:simple_prop_s-open} Let $X=\prod\limits_{t\in T}X_t$, $|X_t|>1$ for all $t\in T$, $x\in X$ and $A\subseteq X$. Then
\begin{enumerate}
  \item \label{S-Clopen} $A\in\mathcal S(X)$ if and only if $X\setminus A \in \mathcal S(X)$;

  \item\label{pr:union-snow}  $A\in\mathcal S(X)$ if and only if $A=\bigcup\limits_{x\in A}\sigma(x)$;

  \item\label{prop:componenta} $\sigma(x)$ is the smallest $\mathcal S$-open set which contains $x$;

  \item\label{s_open_dense}  if $A\in \mathcal S(X)$, then $A$ is dense in $(X,\tau)$.

  \item  there exists a non-trivial $\mathcal S$-open subset of $X$ if and only if  $|T|\ge\aleph_0$.
\end{enumerate}
\end{proposition}

It follows from Proposition~\ref{prop:simple_prop_s-open} that $\sigma$-products of two distinct points of $\prod\limits_{t\in T}X_t$ either coincide, or does not intersect. Consequently, the family of all $\sigma$-products of an arbitrary $\mathcal S$-open set $X\subseteq \prod\limits_{t\in T}X_t$ generates a partition of $X$ on  mutually disjoint $\mathcal S$-open sets, which we will call {\it $\mathcal S$-components of $X$.}

Notice that every $\mathcal S$-component of $X$ is an indiscrete subspace of $(X,\mathcal S(X))$ and the space $(X,\mathcal S(X))$ is a topological sum of indiscrete spaces.

Proposition \ref{prop:simple_prop_s-open} easily implies the following properties of the $\mathcal S$-topology.

\begin{proposition} Let $X=\prod\limits_{t\in T}X_t$ and $|X_t|>1$ for all $t\in T$. Then the space $(X,{\mathcal S})$
  \begin{enumerate}
   \item  is first-countable;

   \item  does not satisfy $T_0$, $T_1$ and $T_2$ if $T\ne\emptyset$;

   \item  satisfy $T_3$, $T_{3\frac 12}$ and $T_4$;

   \item\label{ArcwiseConn} is arcwise connected if and only if $|T|<\aleph_0$;

   \item is extremally disconnected.
\end{enumerate}
\end{proposition}

\section{A necessary and sufficient condition for a strongly separately continuous mapping to be continuous}

\begin{definition}\label{def:sep-cont}
  {\rm  Let $(X_t:t\in T)$ be a family of topological spaces, $Y$ be a topological space and let $X\subseteq \prod\limits_{t\in T}X_t$ be an $\mathcal S$-open set. A mapping $f:X\to Y$ is said to be {\it separately continuous at a point $a=(a_t)_{t\in T}\in X$ with respect to the $t$-th variable} provided that the mapping $g:X_t\to Y$ defined by the rule $g(x)=f(a_t^x)$ for all $x\in X_t$ is continuous at the point $a_t\in X_t$.}
\end{definition}

Let $f:X\to Y$ be a mapping between topological spaces $X$ and $Y$, $a\in X$ and $b\in Y$. Then $\lim\limits_{x\to a}f(x)=b$ if and only if for any neighborhood $V$ of $b$ in $Y$ there is a neighborhood $U$ of $a$ in $X$ such that $f(U)\subseteq V$.

\begin{definition}
  {\rm Let $X\subseteq\prod\limits_{t\in T}X_t$ be an $\mathcal S$-open set, $\mathcal T$ be a topology on $X$ and let $(Y,d)$ be a metric space. A mapping $f:(X,\mathcal T)\to Y$ is called {\it strongly separately continuous at a point $a\in X$ with respect to the $t$-th variable} if
  $$
\lim_{x\to a}d(f(x),f(x_{t}^a))=0.
  $$}
  \end{definition}

  \begin{definition}
    {\rm A mapping $f:X\to Y$ is
    \begin{itemize}
      \item {\it (strongly) separately continuous at a point $a\in X$} if $f$ is (strongly) separately continuous at  $a$ with respect to each variable $t\in T$;

    \item {\it (strongly) separately continuous on the set $X$} if $f$ is (strongly) separately continuous  at each point $a\in X$ with respect to each variable $t\in T$.
        \end{itemize} }
  \end{definition}

\begin{definition}
  {\rm A set $A\subseteq \prod\limits_{t\in T}X_t$ is called {\it projectively symmetric with respect to a point $a\in A$} if for all $t\in T$
  and for all $x\in A$ we have $x_{t}^a\in A$.}
\end{definition}

\begin{definition}
  {\rm Let $X\subseteq \prod\limits_{t\in T}X_t$ and $\mathcal T$ be a topology on $X$. Then $(X,\mathcal T)$ is said to be {\it locally projectively symmetric} if every $x\in X$ has a base of projectively symmetric neighborhoods with respect to $x$.}
\end{definition}

It is easy to see that an arbitrary $\mathcal S$-open subset of a product $\prod\limits_{t\in T}X_t$ of topological spaces $X_t$ equipped either with the Tychonoff topology $\tau$, or with the $\mathcal S$-topology, is a locally projectively symmetric space. All classical spaces of sequences as the space $c$ of all convergence sequences or the spaces $\ell_p$ with $0<p\le \infty$ are locally projectively symmetric.

\begin{proposition}
Let $X$ be an $\mathcal S$-open subset of a product $\prod\limits_{t\in T}X_t$ of topological spaces equipped with a locally projectively symmetric topology $\mathcal T$, $a=(a_t)_{t\in T}\in X$, $Y$ be a metric space and let $f:(X,\mathcal T)\to Y$ be a continuous mapping at the point $a$. Then  $f$ is strongly separately continuous  at $a$.
\end{proposition}

\begin{proof}
Fix $\varepsilon>0$ and $t\in T$. Take a projectively symmetric with respect to $a$ neighborhood $U$ of $a$ such that $$d(f(x),f(a))<\frac{\varepsilon}{2}$$ for all $x\in U$. Then   $x_{t}^a\in U$ and
 $$
 d(f(x),f(x_{t}^a))\le d(f(x),f(a))+d(f(a),f(x_{t}^a))<\frac{\varepsilon}{2}+\frac{\varepsilon}{2}=\varepsilon
 $$
 for all $x\in U$.
\end{proof}

Note that the inverse implication is not true in case $(X,\mathcal T)$ is not a locally projectively symmetric space, as the example below shows.

\begin{example}
   Let $X$ be the Niemytski plane, i.e. $X=\mathbb R\times [0,+\infty)$, where a base of neighborhoods of $(x,y)\in X$ with $y>0$ form open balls with the center in $(x,y)$, and a base of neighborhoods of $(x,0)$ form the sets $U\cup\{(x,0)\}$, where $U$ is an open ball which tangents to $\mathbb R\times \{0\}$ in the point $(x,0)$. Then there exists a continuous function $f:X\to\mathbb R$, which is not strongly separately continuous.
\end{example}

\begin{proof}
  Denote $a=(0,0)$ and consider an arbitrary basic neighborhood $U$ of $a$ in $X$. Since $X$ is a completely regular space, there exists a continuous function $f:X\to [0,1]$ such that $f(a)=0$ and $f(x,y)=1$ for all $(x,y)\in X\setminus U$. We show that $f$ is not strongly separately continuous at the point $a$ with respect to the second variable. Indeed, fix $\varepsilon=\frac 12$  and a neighborhood $V$ of $a$. By the continuity of $f$ at $a$ we choose a neighborhood $V_0$ of $a$ such that $V_0\subseteq V$ and $f(x,y)<\frac 12$ for all $(x,y)\in V_0$. Now let $(x,y)\in V_0$ be an arbitrary point with $x\ne 0$. Then $f(x,0)=1$ and $|f(x,y)-f(x,0)|=1-f(x,y)>\frac 12$. Hence, $f$ is not strongly separately continuous.
\end{proof}

Let $\mathcal E$ denote the Euclidean topology on $\mathbb R$, $X_1=X_2=(\mathbb R, \mathcal E)$ and let $\mathcal T$ be a discrete topology on the set $X=\mathbb R\times \mathbb R$. We consider the following function $f:X\to\mathbb R$,
$$
f(x_1,x_2)=\left\{\begin{array}{ll}
                    1, & (x_1,x_2)\in X\setminus\{(0,0)\}, \\
                    0, & (x_1,x_2)=(0,0).
                  \end{array}
\right.
$$
Then $f$ is continuous on $(X,\mathcal T)$, since $X$ is discrete, but $f$ is not separately continuous on $X_1\times X_2$ at the point $(0,0)$ by definition~\ref{def:sep-cont}. Therefore, it is natural to find conditions on a topology on a product of topological spaces such that the implication ''$f$ is continuous $\Rightarrow$ $f$ is separately continuous'' holds.

\begin{definition}\label{def:top-coordinated}
  {\rm Let $X\subseteq \prod\limits_{t\in T}X_t$ be an $\mathcal S$-open subset of a product of topological spaces $(X_t,\mathcal T_t)$ and let $\mathcal T$ be a topology on $X$. We say that $\mathcal T$ is {\it coordinated with the family $(\mathcal T_t:t\in T)$} if
  \begin{gather}\label{eq:convergence}
    \lim\limits_{x\to a_t} a_t^x=a
  \end{gather}
  for all $t\in T$ and $a=(a_s)_{s\in T}\in X$.}
\end{definition}

\begin{proposition}\label{prop:cont-impl-sep}
  Let $\mathcal T$ be a topology on an $\mathcal S$-open subset $X$ of a product of topological spaces $(X_t,\mathcal T_t)$ coordinated with the family $(\mathcal T_t:t\in T)$, $Y$ is a (metric) space and $f:X\to Y$ be a (strongly separately) continuous mapping at a point $a\in X$.
  Then $f$ is separately continuous at $a$.
\end{proposition}

\begin{proof} Fix $t\in T$. Assume that $f$ is continuous at the point $a$ and denote $g(x)=f(a_t^x)$ for all $x\in X_t$. Then
$\lim\limits_{x\to a_t} g(x)=\lim\limits_{x\to a_t} f(a_t^x)=f(a)=g(a_t)$.

We argue similarly in the case $Y$ is a metric space and $f$ is strongly separately continuous.
\end{proof}

\begin{proposition}\label{prop:coordinated-ex}
  Let $X\subseteq \prod\limits_{t\in T}X_t$ be an $\mathcal S$-open subset of a product of topological spaces $(X_t,\mathcal T_t)$ and let $\mathcal T$ be a topology on $X$. If one of the following conditions holds
  \begin{enumerate}
    \item $(X,\mathcal T)=(X,\tau)$, or

    \item $(X_t,\mathcal T_t)=(\mathbb R,\mathcal E)$ for every $t\in T$ and $(X,\mathcal T)$ is a topological vector space,
  \end{enumerate}
  then $\mathcal T$ is coordinated with $(\mathcal T_t:t\in T)$.
\end{proposition}

\begin{proof}
\begin{enumerate}
  \item It immediately implies from the definition of the Tychonoff topolo\-gy.

  \item Fix $a=(a_s)_{s\in T}\in X$ and $t\in T$. Without loss of generality we may assume that $a_s=0$ for all $s\in S$. Let $b=1$. Then $a_t^b\in \sigma_1(a)\subseteq X$. Since the function $\varphi:(X,\mathcal T)\times \mathbb R\to (X,\mathcal T)$, $\varphi(y,\lambda)=\lambda y$, is continuous, we have that $\lim\limits_{\lambda\to 0}\varphi(y,\lambda)=0$ in $(X,\mathcal T)$ for every $y\in X$. Then $\lim\limits_{x\to 0} a_t^x=\lim\limits_{x\to 0} \varphi(a_t^b,x)=0$ in $(X,\mathcal T)$.
\end{enumerate}
\end{proof}

\begin{theorem}\label{prop:strong-top-s}
 Let $X\subseteq\prod\limits_{t\in T}X_t$ be an $\mathcal S$-open set and $(Y,d)$ be a metric space. A mapping \mbox{$f:(X,\mathcal S)\to Y$} is continuous if and only if $f:(X,\mathcal T)\to Y$ is strongly separately continuous for an arbitrary topology $\mathcal T$ on $X$.
\end{theorem}

\begin{proof}
{\it Necessity.} Fix a topology $\mathcal T$ on $X$ and  consider the partition $(\sigma(x_i):i\in I)$ of the set $X$ on $\mathcal S$-components $\sigma(x_i)$. We notice that $f|_{\sigma(x_i)}=y_i$, where $y_i\in Y$ for all $i\in I$, since $f$ is continuous on $(X,\mathcal S)$. Let $a=(a_t)_{t\in T}\in X$ and $t\in T$. If $x\in X$, then $x\in\sigma(x_i)$ for some $i\in I$. Moreover, $x_{t}^a\in\sigma(x_i)$. Then $f(x)=f(x_{t}^a)=y_i$. Hence, $d(f(x),f(x_{t}^a))=0$ for all $x\in X$. Hence, $f$ is strongly separately continuous on $(X,\mathcal T)$.

{\it Sufficiency.} Put $\mathcal T=\mathcal S$. Fix $x_0\in X$ and show that $f$ is continuous at $x_0$ on $(X,\mathcal S)$. Let $x_0\in\sigma(x_i)$ for some $i\in I$.  Since $f$ is strongly separately continuous at $x_0$ and $\sigma(x_0)=\sigma(x_i)$,  we have $d(f(x),f(x_{t}^{x_0}))=0$ for all $x\in\sigma(x_i)$ and $t\in T$. Consequently, $f(x)=f(x_0)$ for all $x\in \sigma(x_i)$. Since the set $\sigma(x_i)$ is open in $(X,\mathcal S)$, $f$ is continuous at $x_0$.
\end{proof}

Now we give a necessary and sufficient condition for a strongly separately continuous mapping to be continuous.

\begin{theorem}\label{th:suff-cond}
  Let $X\subseteq \prod\limits_{t\in T}X_t$ be an $\mathcal S$-open subset of a product of topological spaces $X_t$, $\mathcal T$ be a topology on $X$ such that $(X,\mathcal T)$ is a locally projectively symmetric space, $(Y,d)$ be a metric space and let $f:(X,\mathcal T)\to Y$ be a strongly separately continuous mapping at the point  $a=(a_t)_{t\in T}\in X$. Then   $f$ is continuous at the point $a$ if and only if
 \begin{gather}
 \begin{gathered}
  \mbox{for every\,\,\,} \varepsilon>0\mbox{\,\,\,there exist a set\,\,\,} T_0\subseteq T \mbox{\,\,\,with\,\,\,}|T_0|<\aleph_0\\
   \mbox{\,\,\, and a neighborhood\,\,\,} U \mbox{\,\,\,of\,\,\,} a \mbox{\,\,\,in\,\,\,} (X,\mathcal T) \mbox{\,\,\,such that}\\
 d(f(a),f(x_{T_0}^{a}))<\varepsilon \mbox{\,\,\, for all\,\,\,} x\in U.
\end{gathered}\label{eq:suff-cond}
 \end{gather}
\end{theorem}

\begin{proof} {\sc Necessity}. Suppose $f$ is continuous at the point $a$ and $\varepsilon>0$. Take a neighborhood $U$ of $a$ such that $d(f(x),f(a))<\varepsilon$ for all $x\in U$ and put $T_0=\emptyset$. Then $x_{T_0}^a=x$, which implies condition (\ref{eq:suff-cond}).

{\sc Sufficiency}.  Fix $\varepsilon>0$. Using the condition of the theorem we take a finite set $T_0\subseteq T$ and a neighborhood $U$ of $a$ in $(X,\tau)$ such that $$
  d(f(a),f(x_{T_0}^{a}))<\frac{\varepsilon}{2}
  $$
  for  every $x\in U$. If $T_0=\emptyset$, then  $d(f(x),f(a))<\varepsilon$ for all $x\in U$. Now assume $T_0=\{t_1,\dots,t_n\}$.  Since $f$ is strongly separately continuous at $a$, for every $k=1,\dots,n$ we choose a neighborhood $V_k$ of the point $a$ such that
  $$
  d(f(x),f(x_{t_k}^{a})<\frac{\varepsilon}{2n}
  $$
  for all $x\in V_k$. We take a projectively symmetric with respect to the point $a$ neighborhood $W$ of $a$ such that
  $$
  W\subseteq U\cap\bigl(\bigcap\limits_{k=1}^n V_k\bigr).
  $$
Since $W$ is projectively symmetric set with respect to $a$, we can show inductively that $x_{\{t_1,\dots,t_k\}}^a\in W$ for every $k=1,\dots,n$ and  for every $x\in W$.  Then for all $x\in W$ we have
\begin{gather*}
  d(f(x),f(a))\le d(f(x),f(x_{T_0}^a))+d(f(x_{T_0}^a),f(a))<\\
  <d(f(x),f(x_{t_1}^a))+d(f(x_{t_1}^a),f(x_{\{t_1,t_2\}}^a))+\dots+\\
  +d(f(x_{\{t_1,\dots,t_{n-1}\}}^a),f(x_{\{t_1,\dots,t_n\}}^a))+\frac{\varepsilon}{2}<\\
  <\frac{\varepsilon}{2n}\cdot n+\frac{\varepsilon}{2}=\varepsilon.
\end{gather*}
Hence, $f$ is continuous at the point $a$.
\end{proof}

We notice that the similar condition given by Visnyai in \cite{TV} for functions defined on $\ell_2$ is stronger than
condition~(\ref{eq:suff-cond}).

The following corollary generalizes the result of Dzagnidze~\cite[Theorem 2.1]{Dzagnidze}.

\begin{corollary}\label{cor:finite-dim-strong-cont}
  Let $X$ be an $\mathcal S$-open subset of a product $\prod\limits_{t\in T} X_t$ of topological spaces $X_t$, $|T|<\aleph_0$ and $(Y,d)$ be a metric space. Then any strongly separately continuous mapping \mbox{$f:(X,\tau)\to Y$} is continuous.
\end{corollary}

\begin{proof}
  Fix an arbitrary point $a\in X$ and a strongly separately continuous mapping \mbox{$f:(X,\tau)\to Y$}. It is easy to see that $f$  satisfy condition~(\ref{eq:suff-cond}). Indeed, for $\varepsilon>0$ we put $T_0=T$ and $U=X$. Then for all $x\in U$ we have $x_{T_0}^a=a$ and consequently $d(f(a),f(x_{T_0}^a))=0<\varepsilon$. Hence, $f$ is continuous at $a$ by Theorem~\ref{th:suff-cond}, provided $(X,\tau)$ is a locally projectively symmetric space.
\end{proof}

The proposition below shows that Corollary~\ref{cor:finite-dim-strong-cont} is not valid for a product of  infinitely many topological spaces.

\begin{proposition}\label{prop:everywhere_discont}
  Let $X=\prod\limits_{t\in T}X_t$ be a product of topological spaces $X_t$, where $|X_t|>1$ for every $t\in T$, let $|T|> \aleph_0$ and $(Y,d)$ be a metric space with $|Y|>1$. Then there exists a strongly separately continuous everywhere discontinuous mapping $f:(X,\tau)\to Y$.
\end{proposition}

\begin{proof}
  Fix $x_0\in X$ and $y_1,y_2\in Y$, $y_1\ne y_2$. According to Proposition~\ref{prop:simple_prop_s-open}, $\sigma(x_0)\ne\emptyset\ne X\setminus\sigma(x_0)$.  Set $f(x)=y_1$ if $x\in\sigma(x_0)$ and $f(x)=y_2$ if $x\in X\setminus \sigma(x_0)$. Then $f$ is everywhere discontinuous on $X$. Indeed, let $a\in X$  and $f(a)=y_1$. Take an open neighborhood $V$ of $y_1$ such that $y_2\not\in V$. If $U$ is an arbitrary neighborhood of $a$ in $(X,\tau)$, then there is $x\in U\setminus\sigma(x_0)$. Therefore, $f(x)=y_2\not\in V$ and $f$ is discontinuous at $a$. Similarly, $f$ is discontinuous at $a$ in the case $f(a)=y_2$.

   Since the set $\sigma(x_0)$ is clopen in $(X,\mathcal S)$, the mapping $f:(X,\mathcal S)\to Y$ is continuous. It remains to apply Theorem~\ref{prop:strong-top-s}.
\end{proof}

\section{A necessary condition on the discontinuity point set}

Let $f:X\to Y$  a mapping between topological spaces. By $C(f)$  we denote the continuity point set of $f$ and by $D(f)$ we denote the discontinuity point set of $f$.

Let ${\mathcal U}_x$ be a system of all neighborhoods of a point $x$ in $X$. For a mapping $f:X\to (Y,d)$ we set
$$
\omega_f(A)=\mathop{\sup}\limits_{x',x''\in A} d(f(x'),f(x''))\quad\mbox{and}\quad \omega_f(x)=\mathop{{\rm inf}}\limits_{U\in{\mathcal U}_x} \omega_f(U).
$$

\begin{definition}
A set $W\subseteq \sigma(a)$ is called {\it nearly open in $(\sigma(a),\tau)$} if for any finite set $T_0\subseteq T$ the set
  $$
  W_{T_0}=\{z\in\prod\limits_{t\in T_0}X_t: a_{T_0}^z\in W\}
  $$
is open in the space $(\prod\limits_{t\in T_0}X_t,\tau)$.
\end{definition}

\begin{theorem}\label{th:necessary_cond}
  Let $(X_t:t\in T)$ be a family of topological spaces, $X=\prod\limits_{t\in T}X_t$, $a=(a_t)_{t\in T}\in X$, $(Y,d)$ be a metric space and $f:(\sigma(a),\tau)\to Y$ be a strongly separately continuous mapping. Then the discontinuity point set $D(f)$ of $f$ is nearly open in $(\sigma(a),\tau)$.
\end{theorem}

\begin{proof} Let $T_0\subseteq T$ be an arbitrary finite set and $Z=\prod\limits_{t\in T_0}X_t$.
  For $z\in Z$ we write $\varphi(z)= a_{T_0}^z$ and $G=(D(f))_{T_0}$.

If $T_0=\emptyset$, then $G=\emptyset$. Now let $T_0=\{t_1,\dots,t_n\}$, $w=(w_t)_{t\in T_0}\in G$, $u=\varphi(w)\in\sigma(a)$ and $\varepsilon=\frac 13\omega_f(u)$. We observe that $\varepsilon>0$, provided $f$ is discontinuous at $u$. Since $f$ is strongly separately continuous at the point $u$, there exists a basic neighborhood $U_0$ of $u$ in $(X,\tau)$ such that
  \begin{gather}\label{gath:th:disc1}
    d(f(x),f(x_{\{t\}}^u))<\frac{\varepsilon}{6n}
  \end{gather}
  for all $t\in T_0$ and $x\in U_0\cap\sigma(a)$. Since the mapping $\varphi:(Z,\tau)\to (\sigma(a),\tau)$ is continuous, there exists a basic neighborhood $W_0$ of $w$ in $(Z,\tau)$ such that $\varphi(W_0)\subseteq U_0$.

We show that
  \begin{gather}\label{gath:th:disc2}
    d(f(x),f(x_{T_0}^{\varphi(z)}))<\frac{\varepsilon}{3}
  \end{gather}
 for any $z\in W_0$ and $x\in U_0\cap\sigma(a)$. Let $x_0=x$ and $x_k=(x_{k-1})_{t_k}^{\varphi(z)}$, $k=1,\dots,n$. Then $x_n=x_{T_0}^{\varphi(z)}$. Moreover, since $x,\varphi(z)\in U_0$, $x_k\in U_0$ for every $k$. It follows from (\ref{gath:th:disc1}) that
   \begin{gather*}
     d(f(x_{k-1}),f((x_{k-1})_{t_k}^u))<\frac{\varepsilon}{6n}\quad\mbox{i}\quad d(f(x_{k}),f((x_{k})_{t_k}^u))<\frac{\varepsilon}{6n}
   \end{gather*}
  for every $k=1,\dots,n$. Taking into account the equality $(x_{k-1})_{t_k}^u=(x_{k})_{t_k}^u$, we obtain
   $$
   d(f(x_{k-1}),f(x_k))<\frac{\varepsilon}{3n}.
   $$
    Hence,
    \begin{gather*}
      d(f(x),f(x_{T_0}^{\varphi(z)}))=d(f(x_0),f(x_n))\le \sum\limits_{k=1}^n d(f(x_{k-1}),f(x_k))<\frac{\varepsilon}{3}.
    \end{gather*}

Now we prove that
$$
\omega_f(\varphi(z))\ge \frac{\varepsilon}{3}
$$
for all $z\in W_0$. Let $z\in W_0$ and $x=\varphi(z)$. Since $\omega_f(u)=3\varepsilon$, there exists a net $(u_\lambda)_{\lambda\in\Lambda}$ of points of $\sigma(a)\cap U_0$ such that  $u_\lambda\mathop{\to}\limits_{\lambda\in\Lambda} u$ in $(X,\tau)$ and $d(f(u_\lambda),f(u))\ge \varepsilon$ for every  $\lambda\in\Lambda$. We notice that $u_{T_0}^x=x$. Therefore, $v_\lambda=(u_\lambda)_{T_0}^x\mathop{\to}\limits_{\lambda\in\Lambda} x$. It follows from (\ref{gath:th:disc2}) that
\begin{gather*}
  d(f(u),f(x))<\frac{\varepsilon}{3}\quad\mbox{i}\quad d(f(u_\lambda),f(v_\lambda))<\frac{\varepsilon}{3}.
\end{gather*}
Hence,
\begin{gather*}
  d(f(x),f(v_\lambda))\ge d(f(u),f(u_\lambda))-d(f(u),f(x))-d(f(u_\lambda),f(v_\lambda))>\frac{\varepsilon}{3},
\end{gather*}
consequently,
$\omega_f(x)\ge\frac{\varepsilon}{3}$. Therefore, $x\in D(f)$. Thus, $W_0\subseteq G$, which implies that  $G$ is open in $(Z,\tau)$.
\end{proof}

\section{A sufficient condition on the discontinuity point set}

The definition of strongly separately continuous mapping easily implies the following properties.

\begin{proposition}\label{th:operations-with-ssc}
  Let $X$ be an $\mathcal S$-open subset of a product $\prod\limits_{t\in T}X_t$ of a family of topological spaces $X_t$ and $\mathcal T$ be a topology on $X$.

  If $f,g:(X,\mathcal T)\to \mathbb R$ are strongly separately continuous mappings at $x_0\in X$, then the mappings
    $f(x)\pm g(x)$, $f(x)\cdot g(x)$, $|f(x)|$, $\min\{f(x),g(x)\}$, $\max\{f(x),g(x)\}$ are strongly separately continuous at  $x_0$.

  If $f(x)=\sum\limits_{n=1}^\infty f_n(x)$ is a sum of a uniformly convergent series of strongly separately continuous at $x_0\in X$ mappings $f_n:(X,\mathcal T)\to \mathbb R$, then  $f$ is strongly separately continuous at  $x_0$.
\end{proposition}

%\begin{remark}
%  {\rm For strongly separately continuous mappings $f$ and $g$ the fraction $\displaystyle \frac{f}{g}$ may not be strongly separately continuous. Indeed, let $X=(\sigma(0),\tau)$ be a subspace of $(\mathbb R^{\aleph_0},\tau)$. Consider sequences $x_n=(\mathop{\underbrace{\frac 1n,0,\dots,0}}\limits_{n-1},n,0,\dots)$ and $y_n=(\mathop{\underbrace{0,\dots,0}}\limits_{n-1},n,0,\dots)$. Clearly, $x_n\to 0$ and $y_n\to 0$ in $X$. Moreover, the sets $F_1=\{x_n:n\in\mathbb N\}$ and $F_2=\{y_n:n\in\mathbb N\}$ are nearly closed in $X$ (i.e., every complement $G_i=X\setminus F_i$ is nearly open in $X$ for $i=1,2$). Then by~
%    }
%\end{remark}

If $(X,\|\cdot\|)$ is a normed space, $a\in X$ and $r>0$, then we write
$$
B(a,r)=\{x\in X:\|x-a\|<r\}\quad\mbox{and}\quad B[a,r]=\{x\in X:\|x-a\|\le r\}.
$$

\begin{theorem}\label{th:prod_unit_balls}
Let $((X_n,\|\cdot\|_n))_{n=1}^\infty$ be a sequence of normed spaces, $a\in \prod\limits_{n=1}^\infty X_n$, $w=(w_n)_{n=1}^\infty\in \sigma(a)$, $(r_n)_{n=1}^\infty$ be a sequence of positive reals and
\begin{gather*}
 W=\Bigl(\prod\limits_{n=1}^\infty B(w_n,r_n)\Bigr)\cap \sigma(a).
\end{gather*}
Then there exists a strongly separately continuous function $f:(\sigma(a),\tau)\to [0,1]$ such that
 \begin{gather*}
  W=D(f)\subseteq f^{-1}(0).
 \end{gather*}
\end{theorem}

\begin{proof} Assume without loss of generality that $a=(0,0,\dots)$. We put $X=(\sigma(a),\tau)$ and for every $n\in\mathbb N $ let
\begin{gather*}
  B_{n}=B(w_n,r_n), \quad  Y_n=X_1\times\dots\times X_n \quad\mbox{and}\\
  d_n(x,y)=\max\limits_{1\le i\le n}\|x_i-y_i\|_i \quad\mbox{for all}\,\,\, x,y\in Y_n.
\end{gather*}
 If $x=(x_n)_{n=1}^\infty\in X$ and $k\in\mathbb N$, then $p_k(x)$ stands for $(x_1,\dots,x_k)$.

For every $x=(x_n)_{n=1}^\infty\in X$ we set $h(x)=\bigl(\frac{1}{r_n}(x_n-w_n)\bigr)_{n=1}^\infty$. Then
$h:X\to X$ is a homeomorphism.

For every $n\in\mathbb N$ define
\begin{gather*}
B_n'=\{x\in X_n:\|x\|_n<1\}, \quad S_{n}=\{x\in X_n:\|x\|_n=1\},\\
A_1=X_1\setminus B_1', \quad A_n=\prod\limits_{i=1}^{n-1}B_{i}'\times (X_n\setminus B_{n}')\mbox{\,\,if\,\,}n\ge 2,\\
C_n=(A_n\times \prod\limits_{i=n+1}^\infty X_i)\cap X,\quad  C=\bigsqcup\limits_{n=1}^\infty C_n.
\end{gather*}
We observe that
$$
W'=h(W)=\Bigl(\prod\limits_{n=1}^\infty B_n'\Bigr)\cap X\quad\mbox{and}\quad C=X\setminus W'.
$$

Now for every $n\in\mathbb N$ we consider a function $f_n:X\to\mathbb R$,
$$
f_n(x)=d_n(p_n(x),Y_n\setminus A_n),
$$
and for all $x\in X$ set
$$
  g(x)=\left\{\begin{array}{ll}
             f_n(x), & \mbox{if\,\,} p_n(x)\in A_n\,\,\mbox{for some\,} n\in\mathbb N, \\
             0, & \mbox{otherwise}.
           \end{array}
  \right.
  $$

{\it Claim 1.} {\it The function $g:X\to\mathbb R$ is strongly separately continuous.}

\begin{proof}
Fix $u\in X$ and  $k\in\mathbb N$. For  $x\in X$ we write $y=x_{k}^{u}$ and estimate the difference $|g(x)-g(y)|$.
If $p_n(x)\in A_n$ and $p_m(y)\in A_m$ for some  $n,m\in\mathbb N$, then, using the equality $f_n(y)=f_m(x)=0$ in the case $n\ne m$, we obtain
\begin{gather*}
  |g(x)-g(y)|\le  |f_n(x)-f_n(y)|+|f_m(x)-f_m(y)|\le \\
  \le |d_{n}({p_{n}}(x),{p_{n}}(y))|+|d_{m}(p_{m}(x),p_{m}(y))|\le 2\|x_k-u_k\|_k.
\end{gather*}
If ${p_n}(x)\in A_{n}$ for some $n$ and $y\not\in C$, then $g(y)=f_n(y)=0$ and
\begin{gather*}
  |g(x)-g(y)|\le |d_{n}({p_{n}}(x),{p_{n}}(y))|\le \|x_k-u_k\|_k.
\end{gather*}
The same estimation is valid if $x\not\in C$ and $y\in C$. Finally, if  $x,y\not\in C$, then $|g(x)-g(y)|=0$.
Hence, for all $x\in X$ we have
  \begin{gather*}
      |g(x)-g(y)|\le  2 \|x_k-u_k\|_k,
  \end{gather*}
which implies that
$$
0\le \lim\limits_{x\to u}|g(x)-g(y)|\le 2 \lim\limits_{x\to u} \|x_k-u_k\|_k=0.
$$
Therefore, $g$ is strongly separately continuous at  $u$.
\end{proof}

{\it Claim 2.} {\it The equality
  \begin{gather}\label{eq:dist_between_a_and_compl}
    d_n(u,Y_n\setminus A_n)=\min\limits_{1\le i\le n}\|u_i-S_{i}\|_i.
  \end{gather}
  holds for all $n\in\mathbb N$ and $u=(u_1,\dots,u_n)\in A_n$.}

\begin{proof} Fix $n\in\mathbb N$, $u=(u_1,\dots,u_n)\in A_n$ and let $B=Y_n\setminus A_n$,  $\rho_i=\|u_i-S_{i}\|_i$ if $i=1,\dots,n$, and $\rho=\min\limits_{1\le i\le n}\rho_i$.

We show that $d_n(u,B)\ge \rho$. If $\|u_n\|_n=1$, then $\rho_n=0=\rho$ and the inequality is obvious. Suppose $\|u_n\|_n>1$, choose any  $y=(y_1,\dots,y_n)$ with $d_n(u,y)<\rho$ and check that $y\in A_n$. Notice that $\|u_i-y_i\|_i<\rho_i$ for every $i=1,\dots,n$. Assume that there exists $1\le i\le n-1$ such that $\|y_i\|_i\ge 1$. Then there exists $b\in S_i$ such that $b=u_i+t(y_i-u_i)$ for some $t\in (0,1]$. Then $\|u_i-b\|_i=t\|u_i-y_i\|_i<\rho_i$, a contradiction. Hence, $y_i\in B_{i}'$ for every $i=1,\dots,n-1$. Similarly we can show that $y_n\not\in B_{n}'$. Hence, $y\in A_n$.

Now we show that $d_n(u,B)\le \rho$. Fix $\varepsilon>0$ and assume that $\rho=\rho_i$ for some $1\le i\le n$. If $i<n$, then there exists $v_i\in S_{i}$ with $\|u_i-v_i\|_i<\rho+\varepsilon$. If $i=n$, then, since $\|u_n-S_{n}\|_n=\|u_n-B_{n}'\|_n$, there exists $v_n\in B_{n}'$ such that $\|u_n-v_n\|_n<\rho+\varepsilon$. Let $v_j=u_j$ for $j\ne i$. Then $v=(v_1,\dots,v_n)\in B$ and $d_n(u,v)<\rho+\varepsilon$, which implies that $d_n(u,B)\le\rho$.

Therefore, $d_n(u,B)=\rho$.
\end{proof}

{\it Claim 3.} {\it $W'\subseteq D(g)$.}

\begin{proof}
   Fix $u=(u_1,\dots,u_n,0,\dots)\in W'$. For every $i=1,\dots,n$ we set $\rho_i=\|u_i-S_{i}\|_i$ and $\rho=\min\limits_{1\le i\le n}\rho_i$. Then $\rho\in (0,1]$. For every $i\ge n+1$ we choose $x_i\in X_i$ such that $\|x_i\|_i=1+\rho$ and consider a sequence $(x^m)_{m=1}^\infty$ such that $$x^m=(u_1,\dots,u_n,0,\dots,0,x_{m+n},0,\dots).$$ Clearly, $x^m\in C_{m+n}$ and $x^m\to u$. Since
   $$\|x_{m+n}-z\|_{m+n}\ge |\|x_{m+n}\|_{m+n}-\|z\|_{m+n}|=\rho$$ for all $z\in S_{m+n}$,
   $$
   g(x^m)=\min\{\rho,\|x_{m+n}-S_{m+n}\|_{m+n}\}=\rho.
   $$
Hence, $\lim\limits_{m\to\infty} g(x^m)=\rho>0=g(u)$, which implies the discontinuity of $g$  at~$u$.
\end{proof}

{\it Claim 4.} {\it $C\subseteq C(g)$.}

\begin{proof}
  Fix $u=(u_n)_{n=1}^\infty\in C$  and $\varepsilon>0$. Let  ${p_n}(u)\in A_n$ for some $n\in\mathbb N$ and consider the case $\|u_n\|_n>1$. Since $\psi:Y_n\to\mathbb R$, $\psi(x_1,\dots,x_n)=\min\limits_{1\le i\le n}\|x_i-S_i\|_i$, is continuous at $p={p_n}(u)$, there exists a neighborhood $U=U_1\times\dots\times U_n$ of $p$ in $Y_n$ such that $|\psi(x)-\psi(u)|<\varepsilon$ for all $x\in U$. Let
  $$
  G=\prod\limits_{i=1}^{n-1}(U_i\cap B_i')\times (U_n\cap (X_n\setminus \overline{B_n'}))\times \prod\limits_{i=n+1}^\infty X_i.
  $$
  Then  $|g(x)-g(u)|=|\psi(x)-\psi(u)|<\varepsilon$ for all $x\in G\cap X$.

Now suppose  $\|u_n\|_n=1$. Then $g(u)=f_n(u)=0$. Set
  $$
  V=\prod\limits_{i=1}^{n-1}B_i'\times B(u_n,\varepsilon)\times\prod\limits_{i=n+1}^\infty X_i.
  $$
Let $x\in V\cap X$. If $\|x_n\|_n\le 1$, then  $g(x)=0$, and if $\|x_n\|_n>1$, then  $$g(x)=\min\limits_{1\le i\le n}\|x_i-S_i\|_i\le \|x_n-S_n\|_n\le \|x_n-u_n\|_n<\varepsilon.$$ Hence, $|g(x)-g(u)|<\varepsilon$ for all $x\in V\cap X$.

Therefore, $u\in C(g)$.
\end{proof}

Claim 3 and Claim 4 imply that $D(g)=W'$. Moreover, $g(x)=0$ for all $x\in W'$.

Consider the function $\varphi:X\to\mathbb R$ such that $\varphi(x)=g(h(x))$ for all $x\in X$.

Fix $n\in\mathbb N$ and $u\in X$. Then $h(x_{n}^u)=(h(x))^{h(u)}_{n}$ for every $x\in X$. We have
\begin{gather*}
 \lim\limits_{x\to u} |\varphi(x)-\varphi(x_{n}^u)|= \lim\limits_{x\to u}|g(h(x))-g(h(x_{n}^u))|=\\
  =\lim\limits_{h(x)\to h(u)}|g(h(x))-g((h(x))^{h(u)}_{n})|=0,
\end{gather*}
since $g$ is strongly separately continuous at $h(u)$ with respect to the $n$-th variable. Hence, $\varphi:X\to\mathbb R$ is strongly separately continuous.

If is easy to see that $\varphi$ is continuous at $x\in X$ if and only if $g$ is continuous at $h(x)\in X$. Hence, $D(\varphi)=h^{-1}(D(g))=h^{-1}(W')=W$. Moreover, $\varphi^{-1}(0)=h^{-1}(g^{-1}(0))\supseteq h^{-1}(W')=W$.

Finally, for every $x\in X$ we put
$$
f(x)=\min\{\varphi(x),1\}.
$$
Note that $D(f)=D(\varphi)=W$ and $f(x)=0$ for all $x\in W$.

It remains to observe that $f:X\to [0,1]$ is strongly separately continuous by Proposition~\ref{th:operations-with-ssc}.
\end{proof}

\begin{theorem}\label{th:suff_cond_for_disc}
Let \mbox{$((X_n,\|\cdot\|_n))_{n=1}^\infty$}  be a sequence of finite-dimensional normed spaces, $a\in \prod\limits_{n=1}^\infty X_n$ and $W\subseteq\sigma(a)$ be a nearly open set. Then there exists a strongly separately continuous function $f:(\sigma(a),\tau)\to [0,1]$ such that \mbox{$D(f)=W$}.
\end{theorem}

\begin{proof} Without loss of generality we can assume that $a=(0,0,\dots)$. For every $n\in\mathbb N$ let
\begin{gather*}
Y_n=X_1\times\dots\times X_n, \quad Z_n=Y_n\times\{0\}\times\{0\}\times\dots,\\
G_n=\{(x_1,\dots,x_n)\in Y_n: (x_1,\dots,x_n,0,\dots)\in W\},
\end{gather*}
and $X=(\sigma(0),\tau)$.

Let $x=(x_n)_{n=1}^\infty$  be an arbitrary point of $W$ such that $x_n=0$ for all $n>N$. Since $W$ is nearly open, for every $k=1,\dots,N$ there exists $r_k(x)>0$  such that $F_1=\prod\limits_{k=1}^N B[x_k,r_k(x)]\subseteq G_N$. Since compact set $K_1=F_1\times\{0\}$ is contained in the open in $Y_{N+1}$ set $G_{N+1}$, we can find $r_{N+1}(x)>0$ such that $$F_2=F_1\times B[0,r_{N+1}(x)]\subseteq G_{N+1}.$$ Now the compactness of  $K_2=F_2\times\{0\}\subseteq W|_{Y_{N+2}}$ implies that there exists $r_{N+2}(x)>0$ such that $$F_3=F_2\times B[0,r_{N+2}(x)]\subseteq G_{N+2}.$$ By repeating this process, we obtain a sequence $(r_n(x))_{n=1}^\infty$ of positive reals such that$$
\Bigl(\prod\limits_{n=1}^{\infty} B[x_n,r_n(x)]\Bigr)\cap X\subseteq W.
$$
Now let $W(x)=\Bigl(\prod\limits_{n=1}^{\infty} B(x_n,r_n(x))\Bigr)\cap X$.

Hence, $W=\bigcup\limits_{x\in W}W(x)$. Since for every $n\in\mathbb N$ the family $(W(x)\cap Z_n:x\in W)$ forms an open covering of $V_n=W\cap Z_n$ in  $Z_n$, there exists a countable set $I_n\subseteq W$ such that the family $(W(x)\cap Z_n:x\in I_n)$ is a covering of $V_n$. Put $I=\bigcup\limits_{n=1}^\infty I_n$. Then
\begin{gather*}
W=\bigcup\limits_{n=1}^\infty V_n=\bigcup\limits_{n=1}^\infty\bigcup\limits_{x\in I_n}(W(x)\cap Z_n)=\bigcup\limits_{x\in I}W(x).
\end{gather*}
Let $I=\{x_m:m\in\mathbb N\}$ and $W_m=W(x_m)$.

According to Theorem~\ref{th:prod_unit_balls} there exists a sequence $(f_m)_{m=1}^\infty$ of strongly continuous functions $f_m:X\to [0,1]$ such that $D(f_m)=W_m\subseteq f_m^{-1}(0)$. The last inclusion implies that every  $f_m$ is lower semi-continuous on $X$. For all $x\in X$ we define
$$
f(x)=\sum\limits_{m=1}^\infty \frac{1}{2^m}f_m(x).
$$
Then Proposition~\ref{th:operations-with-ssc} implies that $f:X\to [0,1]$ is strongly separately continuous function. Moreover, $f$ is lower semi-continuous as a sum of uniformly convergent series of lower semi-continuous functions.

Taking into account that   \mbox{$D(g_1+g_2)=D(g_1)\cup D(g_2)$} for any two lower semi-continuous  functions $g_1$ and $g_2$, we obtain $$D(f)=\bigcup\limits_{m=1}^\infty D(f_m)=W.$$
\end{proof}

Combining Theorems~\ref{th:necessary_cond} and \ref{th:suff_cond_for_disc} we obtain
\begin{theorem}\label{th:charact_discont}
  Let $X=\prod\limits_{n=1}^\infty X_n$ be a product of finite-dimensional normed spaces $X_n$, $a\in X$ and $W\subseteq\sigma(a)$. Then $W$ is the discontinuity point set of some strongly separately continuous function $f:(\sigma(a),\tau)\to \mathbb R$ if and only if $W$ is nearly open in $(\sigma(a),\tau)$.
\end{theorem}

\end{document}